\documentclass[12pt,a4paper]{article}
\usepackage[a4paper,inner=1in,outer=1in,vmargin=1in,marginparwidth=1in]{geometry}
\usepackage{color,graphicx,float}
\usepackage{amsfonts}
\usepackage{amsthm}
\usepackage{amsmath}
\usepackage{amssymb}
\usepackage{mathrsfs}
\usepackage{enumerate}
\usepackage{cite}
\usepackage[all]{xy}
\usepackage{indentfirst}
\numberwithin{equation}{section}
\allowdisplaybreaks
\newtheorem{theorem}{Theorem}[section]

\newtheorem{lemma}[theorem]{Lemma}

\theoremstyle{definition}
\newtheorem{definition}[theorem]{Definition}
\newtheorem{remark}{Remark}

\begin{document}
\title{The stabilization of wave equations with moving boundary}
\author{Lingyang Liu\  and\ Hang Gao}
\date{\centering {{\small{ {}School of Mathematics and Statistics, Northeast Normal University,\\ Changchun 130024,  China}}}} \maketitle
\begin{abstract}
In this paper, we consider the stabilization of wave equations with moving boundary. First, we show the solution behaviour of wave equation with Neumann boundary conditions, that is, the energy of wave equation with mixed boundary conditions
may decrease, increase or conserve depending on the different range of parameter. Second, we prove the wellposedness and stabilization for the wave equation with time delay and moving boundary.

\bigskip

\end{abstract}
\renewcommand{\thefootnote}{\fnsymbol{footnote}}
\footnote[0]{The second author is supported by the NSF of China under grants 11471070, 11771074, and 11371084.}
\footnote[0]{E-mail address: liuly938@nenu.edu.cn (Lingyang Liu),\ hangg@nenu.edu.cn (Hang Gao).}
\section{Introduction and main results}

For any $0<k<1,$ set $l_k(t)=1+kt.$ Denote by $Q^k$ the non-cylindrical domain in $\mathbb R^2:$
$Q^k=\{(x,t)\in\mathbb R^2|0<x<l_k(t),\ t>0\}.$ Given $T>0,$ put $Q^k_T=\{(x,t)\in Q^k|t<T\},$ $\Gamma_L=\{(0,t)\in\mathbb R^2|0<t<T\}$ and $\Gamma_R=\{(l_k(t),t)\in\mathbb R^2|0<t<T\}.$

The main objective of this paper is to investigate the boundary feedback stabilization of wave equations.

\subsection{Stabilization of wave equations}

The first purpose of this paper is to study the stabilization of
the following wave equation with boundary damping:
\begin{eqnarray}\label{e1}
\left\{
\begin{array}{rll}
& u_{tt}-u_{xx}=0&\mbox{ in }Q^k,\\[3mm]
&u_x(0,t)=0,u_x(l_k(t),t)+au_t(l_k(t),t)=0&\mbox{ on }(0,\infty),\\[3mm]
&u(x,0)=u^0,u_t(x,0)=u^1&\mbox{ in } (0,1),
\end{array}
\right.
\end{eqnarray}
where $a\in\mathbb R,$ $(u^0,u^1)$ is any given initial data, $a$ is the control and $u$ is the state variable. Let $r=(1,a)^\top\in\mathbb R^2,$
and
we see that $\displaystyle u_x\big(l_k(t),t\big)+au_t\big(l_k(t),t\big)=\frac{\partial u}{\partial r}\big(l_k(t),t\big).$ Obviously, the right boundary condition for system (\ref{e1}) is the directional derivative of $u$ along the direction $r$.
 We concern the relationship  between the control $a$  and the stability of (\ref{e1}).
To this aim, define the following energy of system (\ref{e1}):
\begin{equation*}
E_1(t)=\frac{1}{2}\int^{l_k(t)}_{0}\big[u^2_t(x,t)+u^2_x(x,t)\big]dx.
\end{equation*}

\begin{definition}
(1) The energy of \eqref{e1} decays only at a rate of $m$-th order polynomials, if there exist constants $c_1,c_2>0$ and an $m$-th order polynomial $\varphi(t)$ such that the energy of \eqref{e1} satisfies
$$
c_1\big[\varphi(t)\big]^{-1}E_1(0)\leq E_1(t)\leq c_2E_1(0)\big[\varphi(t)\big]^{-1},\quad t>0.
$$

(2) The energy of \eqref{e1} decays at a rate which is no less than $m$-th order polynomials, if there exists a constant $C>0$ and an $m$-th order polynomial $\varphi(t),$ such that the energy of \eqref{e1} satisfies
$$
E_1(t)\leq CE_1(0)\big[\varphi(t)\big]^{-1},\quad t>0.
$$

(3) The energy of \eqref{e1} decays at a rate which is no more than $m$-th order polynomials, if there exists a constant $C>0$ and an $m$-th order polynomial $\varphi(t),$ such that the energy of \eqref{e1} satisfies
$$
E_1(t)\geq CE_1(0)\big[\varphi(t)\big]^{-1},\quad t>0.
$$

(4) System \eqref{e1} is said to be exponentially stable, if there exist constants $C,\delta>0$ such that for any given $(u^0,u^1)\in H^1_L\big(0,l(t)\big)\times L^2\big(0,l(t)\big),$ the energy of \eqref{e1} satisfies
$$
E_1(t)\leq CE_1(0)e^{-\delta t},\quad t>0.
$$
\end{definition}

\medskip

We will prove that the energy of system (\ref{e1}) may increase, decrease or conserve, which is dependent on the different range of the control $a$. This result is stated as follows.

\medskip

\begin{theorem}\label{th1}
 Fix $\displaystyle0<k<1,$  and let $\displaystyle a_1=\frac{1-\sqrt{1-k^2}}{k},\ a_2=\frac{1+\sqrt{1-k^2}}{k},\ b_1=k\ \mbox{and}\ b_2=\frac{1}{k}.$ It is easy to check that $ a_1<b_1<b_2<a_2.$
 \vspace{1mm}

\noindent $(1)$ If $a<a_1\ \mbox{or}\ a>a_2,$ then the energy of system \eqref{e1} is increasing. Furthermore, there exist solutions of \eqref{e1} such that the corresponding energy increases only at a polynomial rate.

\vspace{1mm}
\noindent $(2)$ If $a=a_1\ \mbox{or}\ a=a_2,$ then the energy of system \eqref{e1} is conserved.

\vspace{1mm}
\noindent $(3)$ If $a_1<a<a_2,$ then the energy of system \eqref{e1} is decreasing. Moreover, there exist solutions of \eqref{e1} such that the corresponding energy decreases only at a polynomial rate. For $a=b_1$ or $a=b_2,$ the energy of \eqref{e1} decays only at a rate of first order polynomials; for $b_1<a<b_2,$ the energy of \eqref{e1} decays at a rate which is no less than first order polynomials; for $a_1<a<b_1$ or $b_2<a<a_2,$ the energy of \eqref{e1} decays at a rate which is no more than first order polynomials.
\end{theorem}

Stabilization theory has been widely investigated for hyperbolic equations in cylindrical domains and there have been a great number of results (see \cite{s1,f1,V2,V,Rid,t}
and the references therein).
In physical situations, many phenomena evolve in domains whose boundary has moving parts. For instance, consider a heat process in a combustion chamber  attached  a piston, where part of the boundary moves with the motion of the piston (see \cite{III}). Another example is the vibration of an extendible flexible beam with right end  supported by a movable base and  left end  imbedded inside a bearing permitting extension and contraction of the beam (see \cite{wang123}). For the wave equation with moving boundary, qualitative theory results has been obtained in the literature(see, for instance, \cite{w,c2,01,c3} and references therein). However,
very few results  on the stabilization of hyperbolic equations in non-cylindrical domains have been known. To the best of our knowledge,
\cite{III} is the first to treat stabilizability problem of wave equation in a domain with moving boundary. The authors proved that the wave equation with moving boundary is stabilizable with viscous damping and compensation. In \cite{m3}, the author proved that the wave equation in a finite moving domain is stable, when the movement is assumed to move slower than light and periodically.
Further, an optimal feedback stabilization of a string with moving boundary is treated, as the author \cite{m7} showed that, if the movement is not too fast, the energy decays exponentially. Recently, in \cite{m6}, the authors analyzed the stabilization of wave dynamics by moving boundary, while the domain remains bounded, and undergoes phases of expansion and contraction. The stabilization of the wave equation with moving boundary and Dirichlet-Neuman boundary conditions was considered in \cite{m5}, where the energy decays exponentially when the movement move slower than light and periodically. In this paper, we
study the stabilization of the wave equation (\ref{e1}) with moving boundary, and the movement satisfies $l_k(t)=1+kt$, $k\in(0,1)$. 
On the other hand, if the left boundary condition of (\ref{e1}) is replaced by the Dirichlet boundary conditions, we refer to \cite{h} that the authors studied the stabilization of the one-dimensional wave equation with general moving boundary. Although the moving boundary we considered is a special boundary, we give a more explicit energy estimate for system (\ref{e1}).




%

\subsection{Stabilization of wave equation with time delay}
The third objective of this paper is devoted to studying the stabilization of
the following wave equation with time delay:
\begin{eqnarray}\label{es}
\left\{
\begin{array}{rll}
& u_{tt}-u_{xx}=0&\mbox{ in }Q^k,\\[3mm]
&u(0,t)=0&\mbox{ on }(0,\infty),\\[3mm]
&u_x(l_k(t),t)=-\mu_1u_t(l_k(t),t)-\mu_2u_t(l_k(t-\tau),t-\tau)&\mbox{ on }(0,\infty),\\[3mm]
&u(x,0)=u^0,u_t(x,0)=u^1 &\mbox{ in } (0,1), \\[3mm]
&u_t(l_k(t-\tau),t-\tau)=g_0(t-\tau)&\mbox{ on }(0,\tau),
\end{array}
\right.
\end{eqnarray}
where $\mu_1,\mu_2\in\mathbb R,$ $(u^0,u^1,g_0)$ is any given initial value and delay $\tau>0.$

\medskip
Set $H^1_L(0,1)=\{u\in H^1(0,1)|\mbox{the trace}\ u(0)=0\},$

The energy of (\ref{es}) is defined by
\begin{equation}
E_2(t)=\frac{1}{2}\int^{l_k(t)}_{0}\big[u^2_t(x,t)+u^2_x(x,t)\big]dx+\frac{\xi}{2}\int^1_0u^2_t\big(l_k(t-\tau\rho),t-\tau\rho\big)d\rho,
\end{equation}
where $\xi$ is a positive coefficient.

We establish a relationship between stability and the sizes of coefficients $\mu_1,\mu_2$ and time delay $\tau$ for \eqref{es} in this paper. The result is stated as follows.

\medskip

\begin{theorem}\label{th2}
\noindent $(1)$ Let $\displaystyle \frac{1-\sqrt{1-k^2}}{k}<\mu_1<\frac{1+\sqrt{1-k^2}}{k}$. If $\displaystyle\mu_2<\frac{-\left|k\mu_1-1\right|+\sqrt{1-k^2}}{k},$ with some $\tau$(depending on $\mu_1,\mu_2,\xi$), then the energy of system \eqref{es} remains to decrease.

\vspace{2mm}
\noindent $(2)$ For any $\mu_1\in\mathbb R,$ if $\displaystyle\mu_2\geq\frac{\left|k\mu_1-1\right|+\sqrt{1-k^2}}{k}$, then the energy of system \eqref{es} is always increasing with some $\tau$(depending on $\mu_1,\mu_2,\xi$).
\end{theorem}

In the past decades, many authors focus on the stabilization for the wave equation with time delay in cylindrical domains. We mention
\cite{DLP,np,ny} and the references therein for a detail statement. In particular, the stabilization of the one-dimensional wave equation with time delay in cylindrical domain was discussed in \cite{x}, where the wave equation is exponentially stable when $\mu_1>\mu_2$ and the system is unstable when $\mu_1<\mu_2$. Moreover, when $\mu_1=\mu_2$, if $\tau\in(0, 1)$ is rational, then the system is unstable; if $\tau\in(0, 1)$ is irrational,  the system is asymptotically stable.
However, as far as we know, this paper is the first attempt to study the stabilization problem for the wave equation with time delay and moving boundary.
It is more complex to treat the stabilization problem for system (\ref{es}) with moving boundary than the case in cylindrical domain. Moreover,
we observe that, if $\mu_2=0$, then the system (\ref{es}) degenerates to the Dirichlet system without time delay, and the conclusion Theorem \ref{th2} (1) will be the same as that in \cite{h}.

\medskip

The paper is organized as follows. In Section 2, we study  the  well-posedness of the problem \eqref{e1}. In Section 3, we give the proof of Theorem \ref{th1} and some examples. In Sections 4,  we prove that problem \eqref{es} is well-posed. Section 5 is devoted to giving the proof of Theorem \ref{th2}.

\section{The wellposedness of \eqref{e1}}
For preliminary, we give some notations first.

$(1)$ Let $\mathbb N$ denote the set of all positive integers. For any $N\in\mathbb N,$ let $\Omega$ be a domain of $\mathbb R^N$ and $\mu_N$ be the Lebesgue measure in $\mathbb R^N.$ Given $i\in\mathbb N$ and $1\leq i\leq N,$ write

$P(x_1,\cdots,x_{i-1},x_{i+1},\cdots,x_N)=\{(x_1,\cdots,x_{i-1},\xi,x_{i+1},\cdots,x_N)\in\mathbb R^N|\xi\in\mathbb R\}$
and $M_i$ for the set of all points $(x_1,\cdots,x_{i-1},x_{i+1},\cdots,x_N)\in\mathbb R^{N-1}$ such that

$P(x_1,\cdots,x_{i-1},x_{i+1},\cdots,x_N)\bigcap\Omega\neq\emptyset.$

$(2)$ We will denote by $AC_i(\Omega)$ the set of all functions $u$ defined on $\Omega$ with the following property:

if the function $u$ is not absolutely continuous on line $P(x_1,\cdots,x_{i-1},x_{i+1},\cdots,x_N),$ then $\mu_{N-1}(M_i)=0.$

Let us denote by $\displaystyle[\frac{\partial u}{\partial x_i}]$ the classical partial derivative of $u$ with respect to $x_i.$ Since $u$ is absolutely continuous for almost lines $P(x_1,\cdots,x_{i-1},x_{i+1},\cdots,x_N),$ it exists almost everywhere in $\Omega.$

$(3)$ Assume that $\Sigma\subset\mathbb R^N$ is an open set. Put $D(\Sigma)=C^\infty_0(\Sigma)$ and use the symbol $D'(\Sigma)$ to denote its dual. $H^1_{loc}(\Sigma)$ is defined as the space of distributions $\phi$ such that for all $\psi\in D(\mathbb R^N),$ $\psi\phi\in H^1(\Sigma).$

Before proving the main theorem, we introduce three lemmas.
\begin{lemma}[see \cite{AOS}; Page 274]\label{lem1}
Let $u\in L_{1,loc}(\Omega)$ and suppose that its distributional derivative $\displaystyle\frac{\partial u}{\partial x_i}$ is an element of $L_{1,loc}(\Omega),$ then there exists a function $\widetilde u\in AC_i(\Omega)$ which is equal to $u$ almost everywhere in $\Omega.$ Moreover,
\begin{equation*}
[\frac{\partial{\widetilde u}}{\partial x_i}]=\frac{\partial u}{\partial x_i},\ \mbox{almost everywhere in}\ \Omega.
\end{equation*}
\end{lemma}

\begin{lemma}[see \cite{JPN}; Page 957]\label{lem2}
Let $\Sigma\subset\mathbb R^2$ be a domain such that for each $\eta\in\mathbb R,$ $J^\eta:=\{\xi\in\mathbb R|(\xi,\eta)\in\Sigma\}$ and for each $\xi\in\mathbb R,$ $J_\xi:=\{\eta\in\mathbb R|(\xi,\eta)\in\Sigma\}$ are intervals of $\mathbb R.$ Set $K_1=\{\xi\in\mathbb R|\Sigma\cap(\{\xi\}\times\mathbb R)\neq\emptyset\}$ and $K_2=\{\eta\in\mathbb R|\Sigma\cap(\mathbb R\times\{\eta\})\neq\emptyset\},$ then $K_1$ and $K_2$ are nonempty open intervals of $\mathbb R.$ Assume that $\varphi\in H^1_{loc}(\Sigma)$ and $\varphi_{\xi\eta}=0$ in $D'(\Sigma),$ then there exist $f\in H^1_{loc}(K_1)$ and $g\in H^1_{loc}(K_2)$ such that
$$
\varphi(\xi,\eta)=f(\xi)+g(\eta),\ \mbox{almost everywhere in}\ \Sigma.
$$
\end{lemma}

\begin{lemma}[see \cite{JPN}; Page 958]\label{l3}
If $u\in H^1_{loc}(Q^k)$ satisfies $u_{tt}-u_{xx}=0$ in $D'(Q^k),$ then exist $f,g\in H^1_{loc}(\mathbb R)$ such that
$$
u(x,t)=f(t+x)+g(t-x),\ \mbox{almost everywhere in}\ Q^k.
$$
Moreover, $u$ can be continuously extended to $\overline {Q^k},$ the traces of $u$ on each line $\{(x,t)\in Q^k|t=t_0\}$ are in $H^1((0,l_k(t_0)))$ and the traces of $u$ on the boundary $\partial Q^k$ of $Q^k$ are in $H^1_{loc}(\partial Q^k).$
\end{lemma}

Based on the three lemmas above, we have
\begin{theorem}\label{th3}
For any $a\in\mathbb R$ and $(u^0,u^1)\in H^1(0,1)\times L^2(0,1),$ the system \eqref{e1} admits a unique weak solution $u\in H^1_{loc}(Q^k).$ Moreover, there exists $f\in H^1_{loc}(\mathbb R)$ such that
$$
u(x,t)=f(t+x)+f(t-x),\ \mbox{almost everywhere in}\ Q^k.
$$
\end{theorem}
\begin{proof}The whole proof is divided into five part.

\medskip
{ The first step.} By Lemma \ref{l3}, we have $\displaystyle u(x,t)=f(t+x)+g(t+x)\ a.e.\ \mbox{in}\ Q^k$ with $f,g\in H^1_{loc}(\mathbb R).$

\medskip
{ The second step.} Using the Neumann boundary condition $u_x(0,t)=0,$ we get $\displaystyle f'(t)+g'(t)=0$ and so $\displaystyle g=f+c,$ where $c$ is a constant. Take $\displaystyle f=f+\frac{c}{2},$ then $u(x,t)=f(t+x)+f(t-x).$

\medskip
{ The third step.} Using the moving boundary condition $\displaystyle u_x(l_k(t))+au_t(l_k(t),t)=0,$ we assert that $\displaystyle(a+1)f'(t+l_k(t))=(1-a)f'(t-l_k(t)),\ a.e.\ \mbox{in}\ \mathbb R.$ Set $F=(1+l_k)\circ(1-l_k)^{-1},$ then we have
\begin{equation}\label{Third}
(1+a)f'\circ F=(1-a)f'\ a.e.\ \mbox{in}\ \mathbb R.
\end{equation}

{The fourth step.} Based on the initial data, we derive $u^0(x)=f(x)+f(-x)$ and $u^1(x)=f'(x)+f'(-x).$ Hence
\begin{equation}\label{Fourth}
\begin{split}
f(x)&:=\frac{1}{2}\big[\int^x_0u^1(y)dy+u^0(x)\big],\quad \forall x\in[0,1],\\[2mm]
f(-x)&:=\frac{1}{2}\big[-\int^x_0u^1(y)dy+u^0(x)\big],\quad \forall x\in[0,1].
\end{split}
\end{equation}

{ The fifth step.} We are going to extend $f$ from $I_0=[-1,1)$ to $\mathbb R.$ It is clear that $F$ is invertible in $\mathbb R.$ Let $I_n=F^n(I_0)=\big[F^n(-1),F^n(1)\big),$ $n\in\mathbb Z.$ We have $I_n\neq I_m,\ n\neq m$ and $\bigcup\limits_{n\in\mathbb Z}I_n=\mathbb R.$ Since $H^1_{loc}(\mathbb R)$ can be imbedded in $C(\mathbb R),$ by \eqref{Fourth} and the continuity of $f,$ we can extend uniquely $f$ from $I_0$ to $\mathbb R.$ $f$ is unique, Thus $u$ is unique.

\medskip
Case 1. $a=-1.$ From \eqref{Third} we know $f'=0,\ a.e.\ \mbox{in}\ \mathbb R$ and consequently $f=c,\ u=2c,$ for any $c\in\mathbb R.$ Hence $u^0=2c$ and $u^1=0.$ This shows that \eqref{e1} has only a trivial solution at $a=-1.$

\medskip
Case 2. $a=1.$ We see that $f'\circ F=0,\ a.e.\ \mbox{in}\ \mathbb R.$ It follows that $f'\circ F^n=0,\ a.e.\ \mbox{in}\ \mathbb R,\ n\neq0.$ Consequently, solutions of \eqref{e1} appear to be constant in the domain $V:$ $\begin{cases}t+x\geq1\\ t-x\geq1\end{cases}$ or $\begin{cases}t+x\leq-1\\ t-x\leq-1\end{cases}$(see figure \ref{fig1}).
\begin{figure}[htp]
\begin{center}
\includegraphics[width=3in]{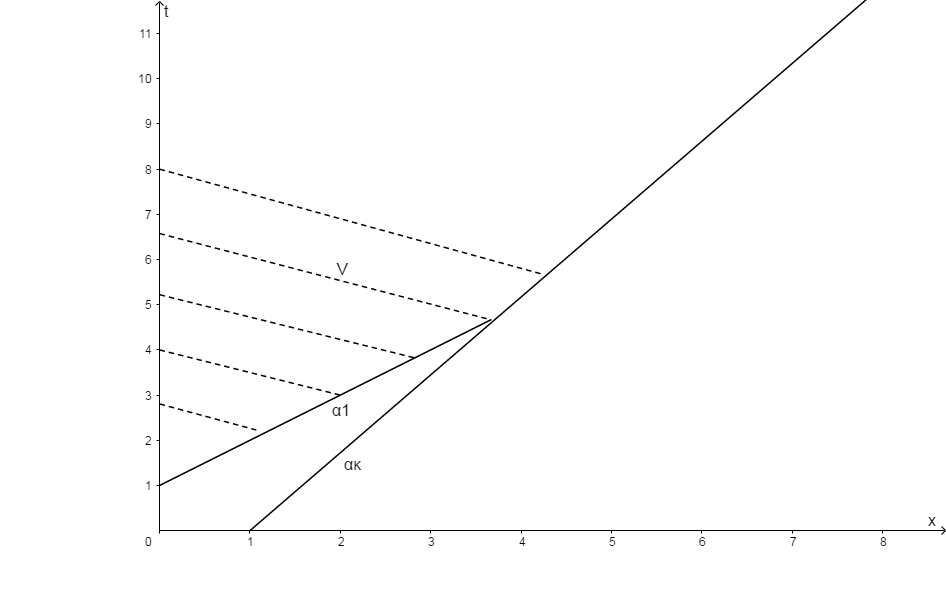}\\
\caption{The graph of case 2}\label{fig1}
\end{center}
\end{figure}

Case 3. $a\neq\pm1.$ From \eqref{Third}, we can derive that
\begin{equation}\label{F2}
f'=\big(\frac{1-a}{1+a}\big)^nf'\circ F^{-n},\ n\in\mathbb Z.
\end{equation}
By \eqref{Fourth}, $f'$ has been known $a.e.$ on $I_0.$ Using the relationship \eqref{F2}, we can obtain $f'$ $a.e.$ on $I_n.$ Furthermore, by integral of $f'$ over $I_n,$ up to a constant $C_n$ (depending on $n$), $f$ is also known on every $I_n$ and $C_0=f(0)=\frac{1}{2}u^0(0).$ Thus, by continuity of $f,$ $f$ is unique in $C(\mathbb R).$ Since $u(x,t)=f(t+x)+f(t-x),\ a.e.\ \mbox{in}\ Q^k,$ $u$ is uniquely determined by initial-boundary condition in $H^1_{loc}(Q^k).$ Finally, we only need to determine $C_n$ for $n\geq1,$ as the argument of $C_n,\ n\leq-1$ is analogous to that.

Due to $I_n:=\big[F^n(-1),F^n(1)\big)=F^n(I_0),$ for any $y\in I_n,$ there exists $x\in I_0,$ such that $F^n(x)=y.$
For $y\in I_1,$ we have $\displaystyle f'(y)=\big(\frac{1-a}{1+a}\big)f'\circ F^{-1}(y)$ by \eqref{F2} and
\begin{equation*}
\begin{split}
f(y)&=\int^y_{F(0)}f'(z)dz+C_1\\[2mm]
&=\big(\frac{1-a}{1+a}\big)\int^y_{F(0)}f'\circ F^{-1}(z)dz+C_1\\[2mm]
&=\big(\frac{1-a}{1+a}\big)\int^y_{F(0)}\frac{1}{\big(F^{-1}(z)\big)'}f'\circ F^{-1}(z)dF^{-1}(z)+C_1.
\end{split}
\end{equation*}
Since $\displaystyle F^{-1}(z)=\frac{(1-k)z-2}{1+k},$ it follows that
\begin{equation*}
f(y)=\big(\frac{1-a}{1+a}\big)\big(\frac{1+k}{1-k}\big)\Big[f\circ F^{-1}(y)-\frac{1}{2}u^0(0)\Big]+C_1,
\end{equation*}
where $F^{-1}(y)\in I_0.$ $f$ has been known on $I_0$ by \eqref{Fourth}. Hence, for any $y\in\big[1,F(0)\big),$ we deduce
\begin{equation}\label{f(y)}
f(y)=\big(\frac{1-a}{1+a}\big)\big(\frac{1+k}{1-k}\big)\Big\{\frac{1}{2}\big[u^0\big(\frac{(k-1)y+2}{1+k}\big)-\int^{\frac{(k-1)y+2}{1+k}}_0\!\!\!u^1(z)dz\big]-\frac{1}{2}u^0(0)\Big\}+C_1.
\end{equation}
From \eqref{Fourth}, we get $\displaystyle f(1)=\frac{1}{2}\big[\int^1_0u^1(z)dz+u^0(1)\big].$ Let $y\rightarrow1$ in \eqref{f(y)}. By continuity of $f,$ it holds that
\begin{equation*}
\big(\frac{1-a}{1+a}\big)\big(\frac{1+k}{1-k}\big)\Big\{\frac{1}{2}\big[u^0\big(1\big)-\int^1_0\!\!\!u^1(z)dz\big]-\frac{1}{2}u^0(0)\Big\}+C_1=\frac{1}{2}\big[\int^1_0u^1(z)dz+u^0(1)\big].
\end{equation*}
This yields that
\begin{equation*}
C_1=\frac{1}{2}\Big[1-\big(\frac{1-a}{1+a}\big)\big(\frac{1+k}{1-k}\big)\Big]u^0(1)+\frac{1}{2}\Big[1+\big(\frac{1-a}{1+a}\big)\big(\frac{1+k}{1-k}\big)\Big]\int^1_0u^1(z)dz+\frac{1}{2}\big(\frac{1-a}{1+a}\big)\big(\frac{1+k}{1-k}\big)u^0(0).
\end{equation*}
For any $y\in I_n,$ we use \eqref{F2} to deduce that
\begin{equation*}
\begin{split}
f(y)&=\int^y_{F^n(0)}f'(z)dz+C_n\\[2mm]
&=\big(\frac{1-a}{1+a}\big)^n\int^y_{F^n(0)}f'\circ F^{-n}(z)dz+C_n\\[2mm]
&=\big(\frac{1-a}{1+a}\big)^n\int^y_{F^n(0)}\frac{1}{\big(F^{-n}(z)\big)'}f'\circ F^{-n}(z)dF^{-n}(z)+C_n.
\end{split}
\end{equation*}
Since $\displaystyle F^{-1}(z)=\frac{(1-k)z-2}{1+k},$ one can get $\displaystyle \big(F^{-n}(z)\big)'=\big(\frac{1-k}{1+k}\big)^n.$ Thus
\begin{equation*}
\begin{split}
f(y)&=\big(\frac{1-a}{1+a}\big)^n\big(\frac{1+k}{1-k}\big)^n\int^y_{F^n(0)}f'\circ F^{-n}(z)dF^{-n}(z)+C_n\\[2mm]
&=\big(\frac{1-a}{1+a}\big)^n\big(\frac{1+k}{1-k}\big)^n\big[f\circ F^{-n}(y)-\frac{1}{2}u^0(0)\big]+C_n.
\end{split}
\end{equation*}
Notice that $F^{-n}(y)\in I_0,$ then for any $y\in\big[F^{n}(-1),F^{n}(0)\big),$
\begin{equation*}
f(y)=\big(\frac{1-a}{1+a}\big)^n\big(\frac{1+k}{1-k}\big)^n\Big\{\frac{1}{2}\big[u^0\big(\!-F^{-n}(y)\big)-\int^{-F^{-n}(y)}_0\!\!\!u^1(z)dz\big]-\frac{1}{2}u^0(0)\Big\}+C_n.
\end{equation*}
If $y\rightarrow F^n(-1),$ then $-F^{-n}(y)\rightarrow1$ and
$$
f(y)\rightarrow\big(\frac{1-a}{1+a}\big)^n\big(\frac{1+k}{1-k}\big)^n\Big\{\frac{1}{2}\big[u^0\big(1\big)-\int^{1}_0\!\!\!u^1(z)dz\big]-\frac{1}{2}u^0(0)\Big\}+C_n.
$$
Similarly, for any $y\in\big[F^{n-1}(0),F^{n-1}(1)\big),$ we have
\begin{equation*}
f(y)=\big(\frac{1-a}{1+a}\big)^{n-1}\big(\frac{1+k}{1-k}\big)^{n-1}\Big\{\frac{1}{2}\big[u^0\big(\!F^{-(n-1)}(y)\big)+\int^{F^{-(n-1)}(y)}_0\!\!\!u^1(z)dz\big]-\frac{1}{2}u^0(0)\Big\}+C_{n-1}.
\end{equation*}
If $y\rightarrow F^{n-1}(1),$ then $F^{-(n-1)}(y)\rightarrow1$ and
$$
f(y)\rightarrow\big(\frac{1-a}{1+a}\big)^{n-1}\big(\frac{1+k}{1-k}\big)^{n-1}\Big\{\frac{1}{2}\big[u^0\big(1\big)+\int^{1}_0\!\!\!u^1(z)dz\big]-\frac{1}{2}u^0(0)\Big\}+C_{n-1}.
$$
By continuity of $f,$ we conclude
\begin{equation*}
\begin{split}
&\big(\frac{1-a}{1+a}\big)^{n-1}\big(\frac{1+k}{1-k}\big)^{n-1}\Big\{\frac{1}{2}\big[u^0\big(1\big)+\int^{1}_0\!\!\!u^1(z)dz\big]-\frac{1}{2}u^0(0)\Big\}+C_{n-1}\\[2mm]
=&\big(\frac{1-a}{1+a}\big)^n\big(\frac{1+k}{1-k}\big)^n\Big\{\frac{1}{2}\big[u^0\big(1\big)-\int^{1}_0\!\!\!u^1(z)dz\big]-\frac{1}{2}u^0(0)\Big\}+C_n,
\end{split}\end{equation*}
which leads to
\begin{equation*}
\begin{split}
C_n=&\frac{1}{2}\big(\frac{1-a}{1+a}\big)^{n-1}\big(\frac{1+k}{1-k}\big)^{n-1}\Bigg\{\Big[1-\big(\frac{1-a}{1+a}\big)\big(\frac{1+k}{1-k}\big)\Big]u^0(1)\\[2mm]
&+\Big[1+\big(\frac{1-a}{1+a}\big)\big(\frac{1+k}{1-k}\big)\Big]\int^1_0u^1(z)dz+\Big[-1+\big(\frac{1-a}{1+a}\big)\big(\frac{1+k}{1-k}\big)\Big]u^0(0)\Bigg\}+C_{n-1}.
\end{split}\end{equation*}
Until now, we complete the proof of Theorem \ref{th3}.
\end{proof}

\begin{remark}
Given $(x,t)\in Q^k,$ put $\xi=t+x$ and $\eta=t-x.$ An easy computation shows that $\xi\in K_1=(0,+\infty)$ and $\eta\in K_2=(-1,+\infty).$ Although we extend $f$ from $I_0$ to $\mathbb R$ using boundary conditions, to prove the existence of solutions in $Q^k,$ it is sufficient to determine a unique $f$ from $I_0$ to $(-1,+\infty).$
\end{remark}

\section{The proof of Theorem \ref{th1}}

Without loss of generality, we assume that functions are sufficiently smooth. Otherwise, we can use the smoothing technique.

\begin{proof}[ Proof of Theorem \ref{th1}] The proof falls naturally into two parts.

1. The stabilization of \eqref{e1}.
Step 1. We study the relationship between the energy of \eqref{e1} and $a.$

As
$$
E_1(t)=\frac{1}{2}\int^{l_k(t)}_{0}\big[u^2_t(x,t)+u^2_x(x,t)\big]dx,
$$
calculating the derivative of $E_1(\cdot)$ with respect to $t,$ we have
\begin{equation*}
\begin{split}
E_1'(t)&=\int^{l_k(t)}_{0}\big[u_t(x,t)u_{tt}(x,t)+u_x(x,t)u_{xt}(x,t)\big]dx\\[2mm]
&\quad+\frac{l'_k(t)}{2}\big[u^2_t(l_k(t),t)+u_x^2(l_k(t),t)\big].
\end{split}
\end{equation*}
Using the first equation and boundary conditions in \eqref{e1}, we arrive at
\begin{eqnarray}
\label{r-1}
\begin{array}{rl}
E_1'(t)\!\!\!&=\displaystyle\int^{l_k(t)}_{0}\big[u_t(x,t)u_x(x,t)\big]_xdx+\frac{l'_k(t)}{2}\big[u_t^2(l_k(t),t)+u_x^2(l_k(t),t)\big]\\[3mm]
&=\displaystyle u_t(l_k(t),t)u_x(l_k(t),t)+\frac{l'_k(t)}{2}\big[u_t^2(l_k(t),t)+u_x^2(l_k(t),t)\big]\\[3mm]
&=\displaystyle u_t(l_k(t),t)\big[-au_t(l_k(t),t)\big]+\frac{k}{2}\big[u_t^2(l_k(t),t)+a^2u_t^2(l_k(t),t)\big]\\[3mm]
&=\displaystyle\frac{ka^2-2a+k}{2}u_t^2(l_k(t),t).
\end{array}
\end{eqnarray}
Let $\displaystyle f(a)=ka^2-2a+k,$ then \eqref{r-1} can be written as
\begin{equation}\label{r0}
E_1'(t)=\frac{f(a)}{2}u_t^2(l_k(t),t).
\end{equation}
It is easy to check that the discriminant of $f(a):$
$$\Delta=4(1-k^2)>0, \ (0<k<1).$$
Hence, two roots are
$$a_1=\frac{1-\sqrt{1-k^2}}{k}\qquad \mbox{and}\qquad a_2=\frac{1+\sqrt{1-k^2}}{k}.$$
When $a<a_1\ \mbox{or}\ a>a_2,$ we conclude that $E_1'(t)>0$ and the energy of \eqref{e1} is increasing; When $ a=a_1\ \mbox{or}\ a=a_2,$ $E_1'(t)=0$ and the energy of \eqref{e1} is conserved; When $a_1<a<a_2,$ $E_1'(t)<0$ and the energy of \eqref{e1} is decreasing.

Step 2. Integrating \eqref{r0} over $(0,T),$ we get
\begin{equation}\label{r1}
E_1(T)-E_1(0)=\frac{f(a)}{2}\int_0^Tu_t^2(l_k(t),t)dt.
\end{equation}
Multiplying the first equation in \eqref{e1} by $xu_x,$ we obtain
\begin{eqnarray*}
\begin{split}
0=(u_{tt}-u_{xx})xu_x=(xu_tu_x)_t-(\frac{1}{2}xu_t^2+\frac{1}{2}xu_x^2)_x+\frac{1}{2}u_t^2+\frac{1}{2}u_x^2.
\end{split}
\end{eqnarray*}
Integrating the above equality on $Q^k_T$ and using the Green's formula, we have
\begin{equation*}
\begin{split}
0&=\int^{l_k(T)}_{0}xu_t(x,T)u_x(x,T)dx-\int^1_{0}xu_t(x,0)u_x(x,0)dx\\[2mm]
&\quad+\int_{\Gamma_{R}}\big[xu_tu_xn_t-\frac{1}{2}x(u^2_t+u^2_x)n_x\big]d\sigma+\int^{T}_{0}E_1(t)dt,
\end{split}
\end{equation*}
where $d\sigma$ is the length element on $\Gamma_{R}$ and $n_t,n_x$ are components of the unit exterior normal $n$ on $\Gamma_R$ corresponding to time and space respectively. It is easy to see that $\displaystyle n=(n_x,n_t)^\top=(\frac{1}{\sqrt{1+k^2}},\frac{-k}{\sqrt{1+k^2}})^\top.$

Notice that $x=l_k(t)$ on $\Gamma_R.$ Transforming the curvilinear integral on $\Gamma_R$ into a single integral about $t,$ using the moving boundary condition  and rearranging the above equality, one gets
\begin{equation}\label{g}
\begin{split}
\int^{T}_0\!\!E_1(t)dt&\!=\!-\int^{l_k(T)}_0xu_t(x,T)u_x(x,T)dx+\int^1_0xu_t(x,0)u_x(x,0)dx\\[1mm]
&\quad+\int_0^T\Big\{l_k(t)u_t(l_k(t),t)u_x(l_k(t),t)k+\frac{l_k(t)}{2}\big[u^2_t(l_k(t),t)+u^2_x(l_k(t),t)\big]\Big\}dt\\[1mm]
&\!=\!-\int^{l_k(T)}_0xu_t(x,T)u_x(x,T)dx+\int^1_0xu_t(x,0)u_x(x,0)dx\\[1mm]
&\quad+\int_0^Tl_k(t)\big[\frac{1}{2}(1+a^2)-ak\big]u^2_t(l_k(t),t)dt.
\end{split}
\end{equation}
Let $\displaystyle g(a)=a^2-2ka+1,$ the discriminant of $g(a)$ is $\displaystyle\Delta=4(k^2-1)<0.$ Thus $\displaystyle g(a)>0,\ \forall a\in\mathbb R.$

Write \eqref{g} as
\begin{equation}\label{r2}
\begin{split}
\int^{T}_0E_1(t)dt&=-\int^{l_k(T)}_0xu_t(x,T)u_x(x,T)dx+\int^1_0xu_t(x,0)u_x(x,0)dx\\[1mm]
&\quad+\int_0^T\frac{g(a)}{2}l_k(t)u^2_t(l_k(t),t)dt.
\end{split}
\end{equation}
On the other hand, multiplying the first equation in \eqref{e1} by $(T-t)u_t,$ one has
\begin{equation*}
\begin{split}
0&=(u_{tt}-u_{xx})(T-t)u_t\\[1mm]
&=\big[\frac{1}{2}(T-t)u_t^2\big]_t+\frac{1}{2}u_t^2-\big[(T-t)u_xu_t\big]_x+(T-t)u_xu_{tx}\\[1mm]
&=\big[\frac{1}{2}(T-t)u_t^2\big]_t+\frac{1}{2}u_t^2-\big[(T-t)u_xu_t\big]_x+\big[\frac{1}{2}(T-t)u_x^2\big]_t+\frac{1}{2}u_x^2\\[1mm]
&=\big[\frac{1}{2}(T-t)u_t^2+\frac{1}{2}(T-t)u_x^2\big]_t-\big[(T-t)u_xu_t\big]_x+\frac{1}{2}u_t^2+\frac{1}{2}u_x^2.
\end{split}
\end{equation*}
Integrating the above equality on $Q^k_T,$ we obtain
\begin{equation}\label{r3}
\begin{split}
-\int_0^T\!\!E_1(t)dt&=\!-\int^{1}_{0}\frac{T}{2}\big[u_t^2(x,0)+u_x^2(x,0)\big]dx\\[1mm]
&\quad+\int_{\Gamma_{R}}\big[\frac{(T-t)}{2}(u_t^2+u_x^2)n_t-(T-t)u_xu_tn_x\big]d\sigma\\[1mm]
&=\!\int_0^T\frac{(T-t)}{2}\Big\{\big[u_t^2(l_k(t),t)+u_x^2(l_k(t),t)\big](-k)-2u_x(l_k(t),t)u_{t}(l_k(t),t)\Big\}dt\\[1mm]
&\quad-TE_1(0)\\[1mm]
&=\!\int_0^T\frac{(T-t)}{2}\big[(a^2+1)(-k)+2a\big]u_t^2(l_k(t),t)dt-TE_1(0)\\[1mm]
&=\!-\int_0^T\frac{f(a)(T-t)}{2}u_t^2(l_k(t),t)dt-TE_1(0).
\end{split}
\end{equation}
Then \eqref{r2} and \eqref{r3} yield that
\begin{equation}\label{r4}
\begin{split}
0
&=\int_0^T\frac{\big[g(a)l_k(t)-f(a)(T-t)\big]}{2}u_t^2(l_k(t),t)dt\\[1mm]
&\quad-\int^{l_k(t)}_0xu_t(x,t)u_x(x,t)dx\bigg|^T_0-TE_1(0)\\[1mm]
&=\int_0^T\Big\{\frac{g(a)-f(a)T}{2}+\frac{\big[kg(a)+f(a)\big]t}{2}\Big\}u_t^2(l_k(t),t)dt\\[1mm]
&\quad-\int^{l_k(t)}_0xu_t(x,t)u_x(x,t)dx\bigg|^T_0-TE_1(0).
\end{split}\end{equation}
Set
\begin{equation*}\begin{split}
h(a)=kg(a)+f(a)&=k(a^2-2ka+1)+ka^2-2a+k\\[1mm]
&=2\big[ka^2-(k^2+1)a+k\big].
\end{split}
\end{equation*}
The discriminant of $h(a)$ is
$$
\Delta=4\big[(k^2+1)^2-4k^2\big]=4(k^2-1)^2>0,\quad (k<1).
$$
Hence $h(a)$ has two roots
$$
b_1=(k^2+1-\sqrt{\Delta})/2k=k,\qquad b_2=(k^2+1+\sqrt{\Delta})/2k=\displaystyle\frac{1}{k}.
$$
It is easy to check that
$$
a_1=\frac{1-\sqrt{1-k^2}}{k}<k<\frac{1}{k}<a_2=\frac{1+\sqrt{1-k^2}}{k}.
$$
{\bf Case 1.} $\displaystyle a=k$ or $a=\displaystyle\frac{1}{k}.$ We have
$$
h(a)= kg(a)+f(a)=0.
$$
Thus
$$
\frac{g(a)}{f(a)}=-\frac{1}{k}.
$$
From \eqref{r4}, we get
\begin{equation}\label{r5}
TE_1(0)+\int^{l_k(t)}_0xu_t(x,t)u_x(x,t)dx\bigg|^T_0=\int_0^T\frac{g(a)-f(a)T}{2}u_t^2(l_k(t),t)dt.
\end{equation}
Substituting \eqref{r1} into \eqref{r5}, we obtain
\begin{equation}\label{r6}
\begin{split}
TE_1(0)+\int^{l_k(t)}_0xu_t(x,t)u_x(x,t)dx\bigg|^T_0
&=\frac{g(a)-f(a)T}{f(a)}\big[E_1(T)-E_1(0)\big]\\[1mm]
&=(-\frac{1}{k}-T\big)[E_1(T)-E_1(0)\big].
\end{split}
\end{equation}
Rearranging the above equality, we arrive at
\begin{equation}\label{kg(a)+f(a)=0}
(1+kT)E_1(T)+k\int^{l_k(T)}_0xu_t(x,T)u_x(x,T)dx=E_1(0)+k\int^{1}_0xu_t(x,0)u_x(x,0)dx.
\end{equation}
Notice that
\begin{equation}\label{inequalityl(t)}
-l_k(t)E_1(t)\leq\int^{l_k(t)}_0xu_t(x,t)u_x(x,t)dx\leq l_k(t)E_1(t),\quad \forall t\geq0,
\end{equation}
which together with \eqref{kg(a)+f(a)=0} yields
$$
(1+kT)E_1(T)-kl_k(T)E_1(T)\leq E_1(0)+kE_1(0),
$$
and
$$
(1+kT)E_1(T)+kl_k(T)E_1(T)\geq E_1(0)-kE_1(0).
$$
As $\displaystyle l_k(t)=1+kt,$ we deduce
$$
(1-k)(1+kT)E_1(T)\leq (1+k)E_1(0),
$$
and
$$
(1+k)(1+kT)E_1(T)\geq (1-k)E_1(0).
$$
Therefore,
\begin{equation}\label{r7}
\frac{(1-k)}{(1+k)(1+kT)}E_1(0)\leq E_1(T)\leq\frac{(1+k)}{(1-k)(1+kT)}E_1(0).
\end{equation}
According to \eqref{r7}, we know that when $\displaystyle a=k$ or $a=\displaystyle\frac{1}{k},$  system \eqref{e1} decays at a rate of first-order polynomials.

{\bf Case 2.} $\displaystyle b_1=k<a<\displaystyle b_2=\frac{1}{k}.$ Now, we have
\begin{equation}\label{case2}
h(a)=kg(a)+f(a)<0.
\end{equation}
By \eqref{case2}, \eqref{r4} leads to
\begin{equation}\label{r8}\begin{split}
TE_1(0)+\int^{l_k(t)}_0xu_t(x,t)u_x(x,t)dx\bigg|^T_0
\leq&\int_0^T\frac{g(a)-f(a)T}{2}u_t^2(l_k(t),t)dt\\[2mm]
=&(\frac{g(a)}{f(a)}-T)\frac{f(a)}{2}\int_0^Tu_t^2(l_k(t),t)dt.
\end{split}\end{equation}
Substituting \eqref{r1} into \eqref{r8}, one gets
\begin{equation*}
\begin{split}
TE_1(0)+\int^{l_k(t)}_0xu_t(x,t)u_x(x,t)dx\bigg|^T_0
\leq\big(\frac{g(a)}{f(a)}-T\big)\big[E_1(T)-E_1(0)\big],\\
\end{split}\end{equation*}
which implies that
\begin{equation*}\begin{split}
&\big(\frac{g(a)}{-f(a)}+T\big)E_1(T)+\int^{l_k(T)}_0xu_t(x,T)u_x(x,T)dx\\[2mm]
\leq&\frac{g(a)}{-f(a)}E_1(0)+\int^{1}_0xu_t(x,0)u_x(x,0)dx.
\end{split}
\end{equation*}
By \eqref{inequalityl(t)}, we derive
\begin{equation*}
\begin{split}
\big[\frac{g(a)}{-f(a)}+T-l_k(T)\big]E_1(T)\leq\big(\frac{g(a)}{-f(a)}+1\big)E_1(0).
\end{split}
\end{equation*}
Notice that $f(a)<0,$ $\forall a\in(b_1,b_2)\subset(a_1,a_2)$ and $l_k(t)=1+kt.$ Finally,
\begin{equation*}
\begin{split}
E_1(T)
\leq\frac{\big[\displaystyle\frac{g(a)}{-f(a)}+1\big]}{\big[\displaystyle\frac{g(a)}{-f(a)}-1+(1-k)T\big]}E_1(0).
\end{split}
\end{equation*}
This means that when $\displaystyle k<a<\displaystyle\frac{1}{k},$ system \eqref{e1} decays at a rate which is no less than first-order polynomials.

{\bf Case 3.} $$a_1=\displaystyle\frac{1-\sqrt{1-k^2}}{k}<a<b_1=k,$$
or
$$b_2=\displaystyle\frac{1}{k}<a<a_2=\frac{1+\sqrt{1-k^2}}{k}.$$
Here, we have
\begin{equation}\label{case3}
h(a)=kg(a)+f(a)>0.
\end{equation}
Applying \eqref{case3} to \eqref{r4} and using \eqref{r1} again, we obtain
\begin{equation*}
\begin{split}
TE_1(0)+\int^{l_k(t)}_0xu_t(x,t)u_x(x,t)dx\bigg|^T_0
\geq&\int_0^T\frac{g(a)-f(a)T}{2}u_t^2(l_k(t),t)dt\\[2mm]
=&(\frac{g(a)}{f(a)}-T)\big[E_1(T)-E_1(0)\big],\\
\end{split}
\end{equation*}
which gives
\begin{equation*}
\begin{split}
&(\frac{g(a)}{-f(a)}+T)E_1(T)+\int^{l_k(T)}_0xu_t(x,T)u_x(x,T)dx\\[2mm]
\geq&\frac{g(a)}{-f(a)}E_1(0)+\int^{1}_0xu_t(x,0)u_x(x,0)dx.
\end{split}
\end{equation*}
By \eqref{inequalityl(t)}, we deduce
\begin{equation*}
\begin{split}
\big[\frac{g(a)}{-f(a)}+T+l_k(T)\big]E_1(T)
\geq(\frac{g(a)}{-f(a)}-1)E_1(0).
\end{split}
\end{equation*}
Since $f(a)<0,$ $\forall a\in(a_1,a_2),$ \eqref{case3} yields $\displaystyle\frac{g(a)}{-f(a)}>\frac{1}{k}>1\ (0<k<1).$ It holds that
\begin{equation*}
\begin{split}
E_1(T)&\geq\frac{\displaystyle(\frac{g(a)}{-f(a)}-1)}{\big[\displaystyle\frac{g(a)}{-f(a)}+1+(1+k)T\big]}E_1(0).
\end{split}
\end{equation*}
This means that when $\displaystyle a_1<a<b_1$ or $\displaystyle b_2<a<a_2,$ system \eqref{e1} decays at a rate which is no more than first-order polynomials.

\medskip
{2. Examples.}
Let's go further to interpret the rate at which the energy of \eqref{e1} decays or grows with some examples. Suppose that \eqref{e1} has a solution in the following form
\begin{equation}\label{sp}
u(x,t)=f(t+\frac{1}{k}+x)+f(t+\frac{1}{k}-x)\quad \mbox{in}\ Q^k.
\end{equation}
Using the moving boundary condition, we get
\begin{equation}\label{spe}
(1+a)f'\big[(1+k)(t+\frac{1}{k})\big]=(1-a)f'\big[(1-k)(t+\frac{1}{k})\big].
\end{equation}
For $a\in\mathbb R$ and $\displaystyle t'=t+\frac{1}{k},$  we have discussed a formula as \eqref{spe} for the wellposedness of \eqref{e1} in Section 2. When $a=1,$ with a similar argument, we claim that the solution will be a constant in the domain $V.$ For $-1<a<1,$ we are going to construct some particular solutions by \eqref{spe}. For simplicity of presentation, set $\displaystyle\mu_a=\frac{1-a}{1+a}$ and $\displaystyle\theta_k=\frac{1+k}{1-k}.$ In addition, set $\displaystyle z=(1-k)(t+\frac{1}{k}),$ obviously, $\displaystyle z>\frac{1}{k}-1,$ $\forall t>0.$ Then \eqref{spe} is converted to
\begin{equation}\label{f'}
f'(\theta_kz)=\mu_af'(z).
\end{equation}
We establish a special function $f'(z)=\displaystyle z^{\displaystyle\frac{\ln\mu_a}{\ln\theta_k}}$ satisfying \eqref{f'}.
\medskip

{\bf Example 1.} If $\displaystyle a=k,$ then $\displaystyle\frac{\ln\mu_a}{\ln\theta_k}=-1$ and $\displaystyle f'(z)=\frac{1}{z}.$ Thus,
$$
f(z)=\ln z+c,
$$
where $c$ is a constant.

By \eqref{sp}, we have
\begin{equation*}\begin{split}
u(x,t)
&=\ln(t+\frac{1}{k}+x)+\ln(t+\frac{1}{k}-x)+2c,\quad t>0,\ 0<x<1+kt.
\end{split}\end{equation*}
Moreover,
\begin{equation*}\begin{split}
u_t(x,t)=\frac{1}{t+\dfrac{1}{k}+x}+\frac{1}{t+\dfrac{1}{k}-x},\quad
u_x(x,t)=\frac{1}{t+\dfrac{1}{k}+x}-\frac{1}{t+\dfrac{1}{k}-x}.
\end{split}\end{equation*}
Put $u^0(x)=u(x,0),$ $u^1(x)=u_t(x,0)$ and then $u$ is a solution to system \eqref{e1} with the initial data $(u^0,u^1).$

Therefore,
\begin{equation*}
\begin{split}
E_1(t)&=\frac{1}{2}\int^{1+kt}_0\big[u_t^2(x,t)+u_x^2(x,t)\big]dx\\[2mm]
&=\int^{1+kt}_0\bigg[\frac{1}{(t+\dfrac{1}{k}+x)^2}+\frac{1}{(t+\dfrac{1}{k}-x)^2}\bigg]dx\\[2mm]
&=\frac{k}{(1+kt)}\big(\frac{1}{1-k}-\frac{1}{1+k}\big),
\end{split}
\end{equation*}
which implies that the energy decays at a rate of first-order polynomials.

{\bf Example 2.} If $\displaystyle a\neq\frac{1}{k},$ then
$$
f(z)=\frac{1}{\displaystyle\frac{\ln\mu_a}{\ln\theta_k}+1}z^{\big(\displaystyle\frac{\ln\mu_a}{\ln\theta_k}+1\big)}+c,
$$
where $c$ is a constant.

By \eqref{sp}, we obtain
\begin{equation*}
\begin{split}
u(x,t)
&=\frac{1}{\displaystyle\frac{\ln\mu_a}{\ln\theta_k}+1}\bigg[\big(t+\frac{1}{k}+x\big)^{\big(\displaystyle\frac{\ln\mu_a}{\ln\theta_k}+1\big)}+\big(t+\frac{1}{k}-x\big)^{\big(\displaystyle\frac{\ln\mu_a}{\ln\theta_k}+1\big)}\bigg]+2c.
\end{split}
\end{equation*}
Moreover,
\begin{equation*}\begin{split}
u_t(x,t)&=\big(t+\frac{1}{k}+x\big)^{\displaystyle\frac{\ln\mu_a}{\ln\theta_k}}+\big(t+\frac{1}{k}-x\big)^{\displaystyle\frac{\ln\mu_a}{\ln\theta_k}},\\[2mm]
u_x(x,t)&=\big(t+\frac{1}{k}+x\big)^{\displaystyle\frac{\ln\mu_a}{\ln\theta_k}}-\big(t+\frac{1}{k}-x\big)^{\displaystyle\frac{\ln\mu_a}{\ln\theta_k}}.
\end{split}\end{equation*}
Given $u^0(x)=u(x,0)$ and $u^1(x)=u_t(x,0),$ $u$ is a solution to system \eqref{e1}.

Therefore,
\begin{equation*}\begin{split}
E_1(t)&=\frac{1}{2}\int^{1+kt}_0\big[u^2_t(x,t)+u^2_x(x,t)\big]dx\\[2mm]
&=\int^{1+kt}_0\bigg[\big(t+\frac{1}{k}+x\big)^{\displaystyle\frac{2\ln\mu_a}{\ln\theta_k}}+\big(t+\frac{1}{k}-x\big)^{\displaystyle\frac{2\ln\mu_a}{\ln\theta_k}}\bigg]dx.
\end{split}\end{equation*}
Further, if $\displaystyle\frac{\ln\mu_a}{\ln\theta_k}=-\frac{1}{2},$ it is easy to check $\displaystyle a=\frac{1-\sqrt{1-k^2}}{k}$ and
\begin{equation*}\begin{split}
E_1(t)&=\int^{1+kt}_0\bigg[\frac{1}{t+\dfrac{1}{k}+x}+\frac{1}{t+\dfrac{1}{k}-x}\bigg]dx\\[2mm]
&=\ln\big(\frac{1+k}{1-k}\big),
\end{split}\end{equation*}
which means that the energy is conserved.

{\bf Example 3.} If $\displaystyle a\neq k$ and $\displaystyle a\neq\frac{1-\sqrt{1-k^2}}{k},$ then
\begin{equation*}\begin{split}
E_1(t)&=\frac{1}{\displaystyle(\frac{2\ln\mu_a}{\ln\theta_k}+1)}\Bigg[\big(t+\frac{1}{k}+x\big)^{\big({\displaystyle\frac{2\ln\mu_a}{\ln\theta_k}}+\displaystyle1\big)}\Bigg|^{1+kt}_0\!\!\!\!-\big(t+\frac{1}{k}-x\big)^{\big({\displaystyle\frac{2\ln\mu_a}{\ln\theta_k}}+\displaystyle1\big)}\Bigg|^{1+kt}_0\Bigg]\\[3mm]
&=\frac{1}{\displaystyle(\frac{2\ln\mu_a}{\ln\theta_k}+\!1)}\Bigg[\big(1+k\big)^{\big({\displaystyle\frac{2\ln\mu_a}{\ln\theta_k}}+\!\displaystyle1\big)}\!\!\!\!-\big(1\!-k\big)^{\big({\displaystyle\frac{2\ln\mu_a}{\ln\theta_k}}+\!\displaystyle1\big)}\Bigg]\big(t+\frac{1}{k}\big)^{\big({\displaystyle\frac{2\ln\mu_a}{\ln\theta_k}}+\!\displaystyle1\big)}.
\end{split}\end{equation*}
For any fixed $k,$ let $\displaystyle g_k(a)=\frac{2\ln\mu_a}{\ln\theta_k}+1=\frac{2\ln\big(\displaystyle\frac{1-a}{1+a}\big)}{\ln\big(\displaystyle\frac{1+k}{1-k}\big)}+1$ denote a function with respect to $a.$ It is easy to find that $g_k(\cdot)$ is a strictly decreasing function for $\displaystyle -1<a<1.$ Moreover,

$(1)$ if $\displaystyle-1<a<\frac{1-\sqrt{1-k^2}}{k},$  $\displaystyle g_k(a)\rightarrow+\infty$ $(a\rightarrow-1)$ and $g_k(a)\rightarrow0$ $(\displaystyle a\rightarrow\frac{1-\sqrt{1-k^2}}{k});$
\medskip

$(2)$ if $\displaystyle\frac{1-\sqrt{1-k^2}}{k}<a<k,$ $g_k(a)\rightarrow0$ $(a\rightarrow\displaystyle\frac{1-\sqrt{1-k^2}}{k})$ and $g_k(a)\rightarrow-1$ $(\displaystyle a\rightarrow k);$
\medskip

$(3)$ if $\displaystyle k<a<1,$ $g_k(a)\rightarrow-1$ $(a\rightarrow k)$ and $g_k(a)\rightarrow-\infty$ $(a\rightarrow1).$

According to the above argument, we conclude that the energy of \eqref{e1} decays or grows only at a polynomial rate.

With regard to $\displaystyle a<-1$ or $a>1,$ by $f'(\theta_kz)=\mu_af'(z),$ we have $\displaystyle f'(\theta^2_kz)=\mu^2_af'(z).$
Similarly, a same conclusion can be obtained.

\medskip
Thus, the proof of Theorem\ref{th1} is finished.
\end{proof}

\section{The wellposedness of \eqref{es}}
\begin{theorem}\label{pro}
Let $\displaystyle0<\tau<\frac{1}{k}.$ For any given $(u^0,u^1,g_0)\in H^1_L(0,1)\times L^2(0,1)\times L^2(-\tau,0),$ system \eqref{es} admits a unique weak solution $u\in H^1_{loc}(Q^k)$ with the form
\begin{equation}\label{A}
u(x,t)=f(t+x)-f(t-x),\quad f\in H^1(\mathbb R),\qquad a.e.\ \mbox{in}\ Q^k.
\end{equation}
\end{theorem}

\begin{proof}
It is easy to see that $u$ in \eqref{A} satisfies the wave equation in the sense of distribution and $u(0,t)=0.$

Next, we shall show that there exists a unique $f\in H^1(\mathbb R)$ except for a constant $f(0),$ such that $u$ satisfies
\begin{equation}\label{abc}
\left\{
\begin{aligned}
&u_x\big(l_k(t),t\big)=-\mu_1u_t\big(l_k(t),t\big)-\mu_2u_t\big(l_k(t-\tau),t-\tau\big), \quad t>0\\[2mm]
&u(x,0)=u^0(x), \quad  u_t(x,0)=u^1(x), \quad x\in(0,1) \\[2mm]
&u_t\big(l_k(t-\tau),t-\tau\big)=g_0(t-\tau), \quad t\in(0,\tau)
\end{aligned}
\right.
\end{equation}
where $(u^0,u^1,g_0)\in H^1_L(0,1)\times L^2(0,1)\times L^2(-\tau,0).$ To do that, we only need to prove that \eqref{abc} determine a unique $f'\in L^2(\mathbb R).$

From \eqref{A}, we have
\begin{eqnarray*}
\begin{array}{rll}
u_x\big(l_k(t),t\big)=f'\big(t+l_k(t)\big)+f'\big(t-l_k(t)\big),\quad a.e.~\mbox{in}~\mathbb R^+,\\[3mm]
u_t\big(l_k(t),t\big)=f'\big(t+l_k(t)\big)-f'\big(t-l_k(t)\big),\quad a.e.~\mbox{in}~\mathbb R^+.
\end{array}
\end{eqnarray*}
Substituting them into the first equation of \eqref{abc}, we get
\begin{align}
&(1+\mu_1)f'\big((1+k)t+1\big)+\mu_2f'\big((1+k)(t-\tau)+1\big)\notag\\[1.5mm]
=&(\mu_1-1)f'\big((1-k)t-1\big)+\mu_2f'\big((1-k)(t-\tau)-1\big), \quad a.e.~\mbox{in}~\mathbb R^+.\tag{$a$}\label{a'}
\end{align}
In the same manner, from the third equation of \eqref{abc}, we have
\begin{align}
&f'\big((1+k)(t-\tau)+1\big)-f'\big((1-k)(t-\tau)-1\big)=g_0(t-\tau), \quad a.e.~\mbox{in}~(0,\tau).\tag{$b$}\label{c'}
\end{align}
At last, from \eqref{A} and the second equation of \eqref{abc}, we see that $\displaystyle u(x,0)=f(x)-f(-x)=u^0(x),$ which derives
\begin{align*}
u_x(x,0)=f'(x)+f'(-x)=u^0_x(x).
\end{align*}
On the other hand,
\begin{align*}
u_t(x,0)=f'(x)-f'(-x)=u^1(x).
\end{align*}
Thus,
\begin{align}
&f'(x)=\dfrac{1}{2}\big[u^0_x(x)+u^1(x)\big], \quad a.e.~\mbox{in}~(0,1)\tag{$c_1$}\label{b'1}\\[2mm]
&f'(-x)=\dfrac{1}{2}\big[u^0_x(x)-u^1(x)\big], \quad a.e.~\mbox{in}~(0,1)\tag{$c_2$}\label{b'2}
\end{align}

In the following, based on \eqref{b'1}, \eqref{b'2}, using \eqref{a'}, \eqref{c'}, we are going to extend $f'$ from the interval $(-1,1)$ to $\big(-(1-k)\tau-1, +\infty\big)$ which is its required extension interval.

First, we extend $f'$ toward the left using \eqref{c'}.

To begin with, we rewrite \eqref{c'} as
\begin{align}\label{c2}
f'\big((1-k)(t-\tau)-1\big)=f'\big((1+k)(t-\tau)+1\big)-g_0(t-\tau), \quad a.e.~\mbox{in}~(0,\tau).\tag{$b'$}
\end{align}
Notice that
\begin{align*}
-(1-k)\tau-1<(1-k)(t-\tau)-1<-1,\quad t\in(0,\tau),\\[2mm]
-(1+k)\tau+1<(1+k)(t-\tau)+1<1,\quad t\in(0,\tau).
\end{align*}

Case 1. $-1\leq-(1+k)\tau+1$ ($\tau\leq\dfrac{2}{1+k}$).

In this case, $f'\big((1+k)(t-\tau)+1\big)$ is known by \eqref{b'1} and \eqref{b'2} because it is located on $(-1,1)$ when $t\in(0,\tau).$
Additionally, $g_0$ is a given function and then functions on the right-hand side of \eqref{c2} have been known. Letting the value of $f'$ on the left-hand side of \eqref{c2} be what we want to extend, we can directly define the value of $f'$ on $\big(-(1-k)\tau-1,-1\big)$ using it.

Case 2. $-(1+k)\tau+1<-1$ ($\tau>\dfrac{2}{1+k}$).

We extend $f'$ in several steps.

Step 1. Let
$$
-1<(1+k)(t-\tau)+1<1.
$$
We get $t\in\big(\tau-\dfrac{2}{1+k},\tau\big).$

As $f'\big((1+k)(t-\tau)+1\big)$ is known from \eqref{b'1} and \eqref{b'2} at this time,
and
$$
-1-\dfrac{2(1-k)}{1+k}<(1-k)(t-\tau)-1<-1, \quad t\in\big(\tau-\dfrac{2}{1+k},\tau\big),
$$
in this step, we can define the value of $f'$ on $\big(-1-\dfrac{2(1-k)}{1+k},-1\big)$ using \eqref{c2}.

Step 2. Let
$$
-1-\dfrac{2(1-k)}{1+k}<(1+k)(t-\tau)+1<-1.
$$
We get $t\in\big(\tau-\dfrac{2}{1+k}-\dfrac{2(1-k)}{(1+k)^2},\tau-\dfrac{2}{1+k}\big).$

Since $f'\big((1+k)(t-\tau)+1\big)$ is known on $\big(-1-\dfrac{2(1-k)}{1+k},-1\big)$ based on the result obtained in Step 1 and for $t\in\big(\tau-\dfrac{2}{1+k}-\dfrac{2(1-k)}{(1+k)^2},\tau-\dfrac{2}{1+k}\big),$

$$
-1-\dfrac{2(1-k)}{1+k}-\dfrac{2(1-k)^2}{(1+k)^2}<(1-k)(t-\tau)-1<-1-\dfrac{2(1-k)}{1+k},
$$
using \eqref{c2} again, we can define the value of $f'$ on
$$\Big(-1-\dfrac{2(1-k)}{1+k}-\dfrac{2(1-k)^2}{(1+k)^2},-1-\dfrac{2(1-k)}{1+k}\Big).$$

Step 3. Let
$$
-1-\dfrac{2(1-k)}{1+k}-\dfrac{2(1-k)^2}{(1+k)^2}<(1+k)(t-\tau)+1<-1-\dfrac{2(1-k)}{1+k},
$$
then
$$
t\in\Big(\tau-\dfrac{2}{1+k}-\dfrac{2(1-k)}{(1+k)^2}-\dfrac{2(1-k)^2}{(1+k)^3},\tau-\dfrac{2}{1+k}-\dfrac{2(1-k)}{(1+k)^2}\Big).
$$
For this $t,$
$$
-1-\dfrac{2(1-k)}{1+k}-\dfrac{2(1-k)^2}{(1+k)^2}-\dfrac{2(1-k)^3}{(1+k)^3}<(1-k)(t-\tau)-1<-1-\dfrac{2(1-k)}{1+k}-\dfrac{2(1-k)^2}{(1+k)^2}.
$$
Therefore, repeating above process until the n-th step, we can already define the value of $f'$ on $\displaystyle\big(-1-\sum^n_{i=1}\frac{2(1-k)^i}{(1+k)^i},-1\big)$ corresponding to $\displaystyle t\in\big(\tau-\sum^n_{i=1}\frac{2(1-k)^{i-1}}{(1+k)^i},\tau),$ using \eqref{c2}. In order to reach the target extension interval $\displaystyle \big(-(1-k)\tau-1,-1\big),$ we shall request $\displaystyle-1-\sum^n_{i=1}\frac{2(1-k)^i}{(1+k)^i}<-(1-k)\tau-1$ for some $n.$

Set $\displaystyle S_n=-1-\sum^n_{i=1}\frac{2(1-k)^i}{(1+k)^i}$ and let $\displaystyle \lim\limits_{n\rightarrow\infty}S_n=-1-\frac{1-k}{k}<-(1-k)\tau-1.$ It follows that $\displaystyle\tau<\frac{1}{k}.$
\medskip

After having done the extension of $f'$ on $\big(-(1-k)\tau-1,-1\big),$ we start to extend $f'$ toward the right using \eqref{a'}.

$(1)$ $\mu_1\neq\pm1.$
\eqref{a'} is equivalent to
\begin{align}
(1+\mu_1)f'\big((1+k)t+1\big)
=&(\mu_1-1)f'\big((1-k)t-1\big)+\mu_2f'\big((1-k)(t-\tau)-1\big) \notag\\[1mm]
&-\mu_2f'\big((1+k)(t-\tau)+1\big)\quad a.e.~\mbox{in}~\mathbb R^+.\tag{$a_1$}\label{a2}
\end{align}

We want functions on the right-hand side of \eqref{a2} to be located in intervals where $f'$ has been known, then the left-hand side of \eqref{a2} is the value of $f'$ that we're going to define.

Case 1. $\tau\leq\dfrac{2}{1-k}<\dfrac{1}{k}$ or $\tau<\dfrac{1}{k}\leq\dfrac{2}{1-k}.$

\noindent In this case,
$$
-1<(1-k)t-1<1,\quad t\in(0,\tau),
$$
$$
-(1-k)\tau-1<-(1+k)\tau+1<(1+k)(t-\tau)+1<1,\quad t\in(0,\tau),
$$
and
\begin{align*}
(1-k)(t-\tau)-1<(1-k)t-1<(1+k)(t-\tau)+1,\quad t>\tau.
\end{align*}

Step 1. Let $\displaystyle(1+k)(t-\tau)+1<1$ $\Rightarrow t<\tau.$
For $t\in(0,\tau),$ functions on the right-hand side of \eqref{a2} are located on $\big(-(1-k)\tau-1,1\big)$ where $f'$ has been known. Hence, using \eqref{a2}, we are able to define the value of $f'$ on $\displaystyle\big(1,(1+k)\tau+1\big).$

Step 2. Let $(1+k)(t-\tau)+1<(1+k)\tau+1$ $\Rightarrow t<2\tau.$ For $t\in(\tau,2\tau),$ using the result obtained in Step 1 and \eqref{a2} again, we can define the value of $f'$ on $\big((1+k)\tau+1,(1+k)2\tau+1\big).$

Step 3. We extend $f'$ with time span $\tau,$ that is, $\mathbb R^+=\bigcup\limits_{n=1}^{+\infty} \big((n-1)\tau,n\tau\big).$ Consequently, we can define $f'$ on $(1,+\infty)= \bigcup\limits_{n=1}^{+\infty} \big((1+k)(n-1)\tau+1,(1+k)n\tau+1\big)$ using \eqref{a2}.

Case 2. $\dfrac{2}{1-k}<\tau<\dfrac{1}{k}.$

In this case, we use two different time spans to extend $f'.$

Step 1. For $t\in\big(0,\dfrac{2}{1-k}\big),$ observing that
$$
-(1-k)\tau-1<(1+k)(t-\tau)+1<(1-k)t-1<1,
$$
at this time, we can define the value of $f'$ on $\big(1,1+\dfrac{2(1+k)}{1-k}\big)$ using \eqref{a2}.

Step 2. If $\displaystyle \frac{2}{1-k}<\tau\leq\frac{2(1+k)}{(1-k)^2},$ then
$$
(1-k)t-1\leq(1+k)(t-\tau)+1,\quad \ t\geq\frac{2}{1-k}+\tau.
$$
Let $\displaystyle(1+k)(t-\tau)+1<1+\dfrac{2(1+k)}{1-k}$$\Rightarrow$ $t<\dfrac{2}{1-k}+\tau.$
Hence, for $t\in\big(\dfrac{2}{1-k},\dfrac{2}{1-k}+\tau\big),$ using \eqref{a2}, we can define the value of $f'$ on $\Big(1+\dfrac{2(1+k)}{1-k},1+\dfrac{2(1+k)}{1-k}+(1+k)\tau\Big).$
In next steps, with fixed time span $\tau.$ we extend $f'$ on $\displaystyle\bigcup\limits^{+\infty}_{n=1}\Big(1,1+\dfrac{2(1+k)}{1-k}+(1+k)n\tau\Big).$

If $\displaystyle\tau>\frac{2(1+k)}{(1-k)^2},$ then
$$
(1+k)(t-\tau)+1<(1-k)t-1,\quad t<\frac{2}{1-k}+\frac{2(1+k)}{(1-k)^2}.
$$
Let $\displaystyle(1-k)t-1<1+\dfrac{2(1+k)}{1-k}$$\Rightarrow$ $\displaystyle t<\frac{2}{1-k}+\frac{2(1+k)}{(1-k)^2}.$
Using \eqref{a2}, we can define the value of $f'$ on $\Big(1+\dfrac{2(1+k)}{1-k},1+\dfrac{2(1+k)}{1-k}+\dfrac{2(1+k)^2}{(1-k)^2}\Big)$
with respect to $\displaystyle t\in \Big(\frac{2}{1-k},\frac{2}{1-k}+\frac{2(1+k)}{(1-k)^2}\Big).$

Step 3. For fixed $\tau,$ there exists a unique positive integer $N,$ such that
$$
\dfrac{2(1+k)^{N-1}}{(1-k)^N} <\tau\leq \dfrac{2(1+k)^N}{(1-k)^{N+1}}.
$$
Let $\displaystyle t_N=\sum\limits^{N}_{i=1}\dfrac{2(1+k)^{i-1}}{(1-k)^i}.$
Moreover,
\begin{align*}
&(1+k)(t-\tau)+1<(1-k)t-1,\quad t<t_N,\\[2mm]
&(1-k)t-1\leq(1+k)(t-\tau)+1,\quad t\geq t_N+\tau.
\end{align*}
For $t<t_N,$ We divide $\displaystyle\big(0,t_N\big)$ to $\displaystyle\bigcup\limits^N_{i=1}\big(t_{n-1},t_n\big),$ where $t_0=0,$ $\displaystyle t_n=\sum\limits^{n}_{i=1}\dfrac{2(1+k)^{i-1}}{(1-k)^i}.$ For every $\big(t_{n-1},t_n\big),$ $1\leq n\leq N,$ similar to Step 1 or the second part of Step 2, we can define the value of $f'$ on $\displaystyle\big(1+\sum\limits^{n-1}_{i=1}\dfrac{2(1+k)^{i}}{(1-k)^i}, 1+\sum\limits^{n}_{i=1}\dfrac{2(1+k)^{i}}{(1-k)^i}\big),$ using \eqref{a2}. For $t>t_N,$ let the time span is $\tau,$ that is, $t_n=t_{n-1}+\tau,$ $n>N.$ For every $(t_{n-1},t_n),$ $n>N,$ using \eqref{a2} again, we can define the value of $f'$ on
$$
\Big(1+\sum\limits^{N}_{i=1}\dfrac{2(1+k)^{i}}{(1-k)^i}+(1+k)(n-1-N)\tau,1+\sum\limits^{N}_{i=1}\dfrac{2(1+k)^{i}}{(1-k)^i}+(1+k)(n-N)\tau\Big).
$$
In summary, for $\displaystyle\mathbb R^+=\bigcup\limits^{+\infty}_{n=1}(t_{n-1},t_n),$ we can define the value of $f'$ on
$\displaystyle\big(1,+\infty\big)$ by \eqref{a2}.
\medskip

$(2)$ $\mu_1=1.$ from \eqref{a'}, we get that for almost everywhere $t>0,$
\begin{align}
2f'\big((1+k)t+1\big)=\mu_2f'\big((1-k)(t-\tau)-1\big)-\mu_2f'\big((1+k)(t-\tau)+1\big).\tag{$a_2$}\label{a3}
\end{align}
Similar to Case 1, with fixed time span $\tau,$ we can define the value of $f'$ on $\displaystyle\big(1,+\infty\big)$ using \eqref{a3}.

$(3)$ $\mu_1=-1.$ from \eqref{a'}, we have that for almost everywhere $t>0,$
\begin{align}
\mu_2f'\big((1+k)(t-\tau)+1\big)-\mu_2f'\big((1-k)(t-\tau)-1\big)=-2f'\big((1-k)t-1\big).\tag{$a_3$}\label{a4}
\end{align}
Combining \eqref{a4} with \eqref{c'}, we deduce that
\begin{align*}
\mu_2g_0(t-\tau)=-2f'\big((1-k)t-1\big) \quad a.e.~\mbox{in}~\mathbb (0,\tau).
\end{align*}
Therefore, $g_0,$ $\big(u^0,u^1\big)$ shall satisfy certain compatibility conditions for $f'$ to be well defined. Since the remainder of the proof is very similar to what we discussed earlier, we omit it.
\end{proof}

\section{Proof of Theorem \ref{th2}}

\begin{proof}

It is easy to check that
\begin{align}\label{E'3}
E'_2(t)=&\big[\frac{k}{2}(1+\mu^2_1)-\mu_1+\frac{\xi}{2\tau}\big]u^2_t\big(l_k(t),t\big)+\big(\frac{k}{2}\mu^2_2-\frac{\xi}{2\tau}\big)u^2_t\big(l_k(t-\tau),t-\tau\big)\notag\\[2mm]
&+\big(k\mu_1\mu_2-\mu_2\big)u_t\big(l_k(t),t\big)u_t\big(l_k(t-\tau),t-\tau\big).
\end{align}
We estimate it in three classifications, using the Cauchy's inequality.
\medskip

$(1)$ $k\mu_1-1=0,\ i.e.\  \mu_1=\dfrac{1}{k}.$ Substituting it into \eqref{E'3}, we get

$$
E'_2(t)=\big(\frac{k}{2}-\frac{1}{2k}+\frac{\xi}{2\tau}\big)u^2_t\big(l_k(t),t\big)+\big(\frac{k}{2}\mu^2_2-\frac{\xi}{2\tau}\big)u^2_t\big(l_k(t-\tau),t-\tau\big).
$$
On the one hand, Let $\displaystyle\frac{k}{2}-\frac{1}{2k}+\frac{\xi}{2\tau}<0$ and $\displaystyle\frac{k}{2}\mu^2_2-\frac{\xi}{2\tau}<0.$ It follows that
\begin{equation}\label{u5}
\left|\mu_2\right|<\frac{\sqrt{1-k^2}}{k},
\end{equation}
and
$$
\frac{\xi}{\dfrac{1}{k}-k}<\tau<\frac{\xi}{k\mu^2_2}.
$$
Under above assumptions, we conclude that there exist a constant $c_1>0,$ such that
$$
E'_2(t)\leq-c_1\Big[u^2_t\big(l_k(t),t\big)+u^2_t\big(l_k(t-\tau),t-\tau\big)\Big],
$$
where $\displaystyle-c_1=\max\Big\{\frac{k}{2}-\frac{1}{2k}+\frac{\xi}{2\tau},\frac{k}{2}\mu^2_2-\frac{\xi}{2\tau}\Big\}.$

On the other hand, Let $\displaystyle\frac{k}{2}-\frac{1}{2k}+\frac{\xi}{2\tau}\geq0$ and $\displaystyle\frac{k}{2}\mu^2_2-\frac{\xi}{2\tau}\geq0.$ we have
\begin{equation}\label{u6}
\left|\mu_2\right|\geq\frac{\sqrt{1-k^2}}{k},
\end{equation}
and
$$
\frac{\xi}{k\mu^2_2}\leq\tau\leq\frac{\xi}{\dfrac{1}{k}-k}.
$$
Thus, there exist a constant $c'_1\geq0,$ such that
$$
E'_2(t)\geq c'_1\Big[u^2_t\big(l_k(t),t\big)+u^2_t\big(l_k(t-\tau),t-\tau\big)\Big],
$$
where $\displaystyle c'_1=\min\Big\{\frac{k}{2}-\frac{1}{2k}+\frac{\xi}{2\tau},\frac{k}{2}\mu^2_2-\frac{\xi}{2\tau}\Big\}.$

$(2)$ $k\mu_1-1<0,\ i.e.\  \mu_1<\dfrac{1}{k}.$ In this case, using the Cauchy's inequality to amplify the right hand of \eqref{E'3}, we get
\begin{align*}
E'_2(t)\leq&\big[\frac{k}{2}(1+\mu^2_1)-\mu_1+\frac{\xi}{2\tau}\big]u^2_t\big(l_k(t),t\big)+\big(\frac{k}{2}\mu^2_2-\frac{\xi}{2\tau}\big)u^2_t\big(l_k(t-\tau),t-\tau\big)\notag\\[2mm]
&+\frac{1}{2}\big(1-k\mu_1\big)\left|\mu_2\right|\Big[u^2_t\big(l_k(t),t\big)+u^2_t\big(l_k(t-\tau),t-\tau\big)\Big].
\end{align*}
Let
\begin{equation*}
\left\{
\begin{aligned}
&A\big(k,\mu_1,\mu_2,\xi,\tau\big)=\frac{k}{2}(1+\mu^2_1)-\mu_1+\frac{1}{2}\big(1-k\mu_1\big)\left|\mu_2\right|+\frac{\xi}{2\tau}<0,\notag\\[2mm]
&B\big(k,\mu_1,\mu_2,\xi,\tau\big)=\frac{k}{2}\mu^2_2+\frac{1}{2}\big(1-k\mu_1\big)\left|\mu_2\right|-\frac{\xi}{2\tau}<0,
\end{aligned}
\right.
\end{equation*}
that is
$$
\frac{k}{2}\mu^2_2+\frac{1}{2}\big(1-k\mu_1\big)\left|\mu_2\right|<\frac{\xi}{2\tau}<\mu_1-\frac{k}{2}(1+\mu^2_1)-\frac{1}{2}(1-k\mu_1)\left|\mu_2\right|.
$$
It may be true if
$$
\frac{k}{2}\mu^2_2+\frac{1}{2}\big(1-k\mu_1\big)\left|\mu_2\right|<\mu_1-\frac{k}{2}(1+\mu^2_1)-\frac{1}{2}(1-k\mu_1)\left|\mu_2\right|.
$$
Thus
\begin{equation}\label{u1}
\dfrac{1-\sqrt{1-k^2}}{k}<\mu_1<\frac{1}{k}, \qquad \left|\mu_2\right|<\dfrac{(k\mu_1-1)+\sqrt{1-k^2}}{k},
\end{equation}
and then
$$
\dfrac{\xi}{2\mu_1-k(1+\mu^2_1)-(1-k\mu_1)\left|\mu_2\right|}<\tau<\dfrac{\xi}{k\mu^2_2+(1-k\mu_1)\left|\mu_2\right|}.
$$
Therefore, there exist a constant $c_2>0,$ such that
$$
E'_2(t)\leq-c_2\Big[u^2_t\big(l_k(t),t\big)+u^2_t\big(l_k(t-\tau),t-\tau\big)\Big],
$$
where
$\displaystyle-c_2=\max\Big\{A\big(k,\mu_1,\mu_2,\xi,\tau\big),B\big(k,\mu_1,\mu_2,\xi,\tau\big)\Big\}.$

\medskip

On the other hand, one has
\begin{align*}
E'_2(t)\geq&\big[\frac{k}{2}(1+\mu^2_1)-\mu_1+\frac{\xi}{2\tau}\big]u^2_t\big(l_k(t),t\big)+\big(\frac{k}{2}\mu^2_2-\frac{\xi}{2\tau}\big)u^2_t\big(l_k(t-\tau),t-\tau\big)\notag\\[2mm]
&+\frac{1}{2}\big(k\mu_1-1\big)\left|\mu_2\right|\Big[u^2_t\big(l_k(t),t\big)+u^2_t\big(l_k(t-\tau),t-\tau\big)\Big].
\end{align*}
Let
\begin{equation*}
\left\{
\begin{aligned}
&C\big(k,\mu_1,\mu_2,\xi,\tau\big)=\frac{k}{2}(1+\mu^2_1)-\mu_1+\frac{1}{2}\big(k\mu_1-1\big)\left|\mu_2\right|+\frac{\xi}{2\tau}\geq0,\notag\\[2mm]
&D\big(k,\mu_1,\mu_2,\xi,\tau\big)=\frac{k}{2}\mu^2_2+\frac{1}{2}\big(k\mu_1-1\big)\left|\mu_2\right|-\frac{\xi}{2\tau}\geq0.
\end{aligned}
\right.
\end{equation*}
We get
\begin{equation}\label{u2}
\mu_1<\frac{1}{k},\qquad \left|\mu_2\right|\geq\dfrac{(1-k\mu_1)+\sqrt{1-k^2}}{k},
\end{equation}
and
$$
\dfrac{\xi}{k\mu^2_2+(k\mu_1-1)\left|\mu_2\right|}\leq\tau\leq\dfrac{\xi}{2\mu_1-k(1+\mu^2_1)-(k\mu_1-1)\left|\mu_2\right|}.
$$
Thus, there exist a constant $c'_2\geq0,$ such that
$$
E'_2(t)\geq c'_2\Big[u^2_t\big(l_k(t),t\big)+u^2_t\big(l_k(t-\tau),t-\tau\big)\Big],
$$
where $\displaystyle c'_2=\min\Big\{C\big(k,\mu_1,\mu_2,\xi,\tau\big),D\big(k,\mu_1,\mu_2,\xi,\tau\big)\Big\}.$

$(3)$ $k\mu_1-1>0,\ i.e.\  \mu_1>\dfrac{1}{k}.$ With an argument similar to that in $(1),$ we get that on one side,

\begin{equation}\label{u3}
\frac{1}{k}<\mu_1<\frac{1+\sqrt{1+k^2}}{k},\qquad \left|\mu_2\right|<\dfrac{(1-k\mu_1)+\sqrt{1-k^2}}{k},
\end{equation}
and
$$
\dfrac{\xi}{2\mu_1-k(1+\mu^2_1)-(k\mu_1-1)\left|\mu_2\right|}<\tau<\dfrac{\xi}{k\mu^2_2+(k\mu_1-1)\left|\mu_2\right|}.
$$
There exist a constant $c_3>0$ (depending on $\mu_1,\mu_2,\xi,\tau$) such that
$$
E'_2(t)\leq-c_3\Big[u^2_t\big(l_k(t),t\big)+u^2_t\big(l_k(t-\tau),t-\tau\big)\Big].
$$
On the other side,
\begin{equation}\label{u4}
\mu_1>\frac{1}{k},\qquad \left|\mu_2\right|\geq\dfrac{(k\mu_1-1)+\sqrt{1-k^2}}{k},
\end{equation}
and
$$
\dfrac{\xi}{k\mu^2_2+(1-k\mu_1)\left|\mu_2\right|}\leq\tau\leq\dfrac{\xi}{2\mu_1-k(1+\mu^2_1)-(1-k\mu_1)\left|\mu_2\right|}.
$$
There exist a constant $c'_3\geq0$ (depending on $\mu_1,\mu_2,\xi,\tau$) such that
$$
E'_2(t)\geq c'_3\Big[u^2_t\big(l_k(t),t\big)+u^2_t\big(l_k(t-\tau),t-\tau\big)\Big].
$$

From the above argument, we can see that the size of coefficients $\mu_1,$ $\mu_2$ and the value of $\tau$ have an impact on stability of the system. Comparing \eqref{u5},\eqref{u1},\eqref{u3} and \eqref{u6},\eqref{u2},\eqref{u4}, respectively, we find that

(F1) Stabilization of the system requires that $\mu_1$ belongs to the range where the system is stable without time delay and the coefficient of time-delay term $\mu_2$ cannot be large.

(F2) No matter what $\mu_1$ is, the system  loses stability if the time-delay term coefficient $\mu_2$ is too large.
\end{proof}

\begin{remark}
In the process of estimating $\eqref{E'3}$ above, we made the coefficients both negative or both positive. In fact, if the coefficient is a plus and a minus, it is complicated, may be a complex fluctuation in itself.
\end{remark}




\end{document}